 \documentclass{amsart}

  \let\oldchapter=\chapter
\def\resetpage{\setcounter{page}{1}}
\def\chapter{\expandafter\resetpage\oldchapter}  
 
\usepackage{mathdots}

\usepackage{bm}
\usepackage{bigints}
\usepackage{epsfig}
\usepackage{epstopdf}

\usepackage{srcltx,amssymb}

\usepackage{graphicx}

    \usepackage{amsthm}
      \usepackage{amsmath}
    \usepackage{amscd}
    \usepackage[title]{appendix}
    \usepackage{slashed}
    \numberwithin{equation}{section} 

   \newtheorem{theorem}{Theorem}

   \usepackage{appendix}
   
 \newtheorem{lemma}[theorem]{Lemma}

    \theoremstyle{definition}

   \numberwithin{theorem}{section}

\begin{document}

        \title{Origin of the uncurling metric formalism from an Inverse Problems regularization method}
    \author{Fred Greensite}

  \begin{abstract} As we have previously shown, recognition of an algebra's group of uncurling metrics leads to a generalization of the usual algebra norm and a novel set of isomorphism invariants.  It is herein detailed how this arose in the analysis of an unusual Inverse Problems regularization method.  \end{abstract}
  
  \maketitle
   
  \markright{How uncurling metrics arose in the context of Inverse Problems} \markleft{Fred Greensite}

 \section{Introduction} Reference \cite{fgreensite:2023} introduces the notions of ``uncurling metrics", ``unital norms", ``inversion characterization groups", and resulting novel isomorphism invariants pertaining to real unital associative algebras.  As detailed here, these concepts originated from a geometrical analysis of an unusual Inverse Problems regularization method.  
   
\section{The finite-dimensional case}
    
 Consider the linear problem described by, \begin{equation}\label{7/7/20.9}  y^\delta=Fx + \nu,\end{equation} where one is given a data vector $y^\delta$ that is a version of an unknown source vector $x$ as transformed by the known linear operator $F:\mathbb{R}^m\rightarrow\mathbb{R}^n$, but where the data vector has been corrupted by a ``noise" vector $\nu$ that is only characterized by \begin{equation}\label{7/12/20.1} \| y^\delta - Fx\|^2 \le \delta^2,\end{equation} where $\delta$ is given and $\|\cdot \|$ is the Euclidean norm.  
    
   If $F$ is an ill-conditioned matrix (e.g., a discretization of a linear compact operator), a solution estimate $F^{-1}y^\delta$ or $F^\dagger y^\delta$ is usually worthless ($F^\dagger$ is the pseudoinverse of $F$).  That is, consider the singular system for $F$ as  $\{u_i ,\sigma_i, v_i\}$, where $\{u_i\}$ is the set of eigenvectors of $FF'$, $\{v_i\}$ is the set of eigenvectors of $F'F$, and $\{\sigma_i^2 \}$ is the set of eigenvalues of $F'F$ (i.e., $\{\sigma_i \}$ is a set of singular values).   The least squares solution estimate of smallest norm is then, \begin{equation}\label{7/7/20.8} x_{\rm poor} = F^\dagger y^\delta = \sum_{ \{i: \sigma_i \ne 0\}} \frac{u_i' y^\delta}{\sigma_i}v_i,\end{equation} where  an expression like $a' b$ indicates the Euclidean inner product of vectors $a$ and $b$.   It is evident that the existence of small singular values  leads to potentially extreme noise amplification, so that the above is a poor estimate of $x$ as regards the at least possibly approachable $x^\dagger \equiv F^\dagger F x$.
 Instead, a common treatment (upon which zero-order Tikhonov regularization with the ``Discrepancy Principle" is based \cite{vogel:2002,engl:2000}) is to supply the solution estimate of minimum norm that is consistent with the modification of (\ref{7/12/20.1}) as \begin{equation}\label{7/7/20.3} \| y^\delta - Fx\|^2 =\delta^2 .\end{equation}   That is, the solution estimate is the point of smallest Euclidean norm on the ``discrepancy ellipsoid",   \begin{equation}\label{7/15/20.1} \partial\Omega_\delta \equiv \{ x\in \mathcal{D}(F): \| y^\delta - Fx\|^2 =\delta^2\},\end{equation} where $\mathcal{D}(F)$ is the domain of $F$.  We use $\mathcal{D}(F)$ here instead of $\mathbb{R}^m$ because we will want (\ref{7/15/20.1}) to also be relevant when $F$ is a compact operator (the general Inverse Problems setting \cite{engl:2000}).
  A Lagrange multiplier argument then gives the solution estimate as, \begin{equation}\label{7/8/20.9} x^\delta_D = (F'F + \gamma I)^{-1}F'y^\delta,\end{equation} where $I$ is the identity operator, and regularization parameter $\gamma$ is the reciprocal of the
   Lagrange multiplier as determined by the requirement that (\ref{7/7/20.3}) be satisfied (the subscript ``$_D$" indicates the estimate follows from a ``deterministic" treatment).  
 
 Alternatively, one may consider both $x$ and $\nu$ to be zero-mean Gaussian random vectors (``signal" and ``noise", respectively), with the covariance matrix of $x$ being $C_x$ (assumed for the moment to be nonsingular), and the covariance matrix of noise vector $\nu$ being $\epsilon^2 I$ ($\delta^2$ in (\ref{7/7/20.3}) might now be viewed as an estimate of the sum of the autocovariances of the components of $\nu$).  The {\it maximum  a posteriori} (Bayesian) estimate  is then, \begin{equation}\label{7/8/20.5} x^\delta_S = (F'F + \epsilon^2 C_x^{-1})^{-1} F'y^\delta,\end{equation}  where the subscript ``$_S$" indicates that the estimate follows from  a ``statistical" treatment \cite{vogel:2002}. 
  
 We can thus associate the deterministic solution estimate (\ref{7/8/20.9}) with a statistical interpretation (\ref{7/8/20.5}), as resulting from assigned signal and noise covariance matrices each proportional to the identity matrix, with $\gamma$ as the inverse of the square of the  ``signal-to-noise" ratio (this inverse square being the ratio of the respective noise and signal component variances).  Thus, the diagonal entries of the noise covariance matrix are assumed to be identical, and these represent the noise variance $\epsilon^2$ in each of the components of the data vector (i.e., the variances of each component of random vector $\nu$).   So the full noise covariance matrix is assumed known {\it a priori}, perhaps based on knowledge of $\delta^2$ under the supposition that $\delta^2 \approx n\epsilon^2$ in (\ref{7/7/20.3}), where $n$ is the dimension of the range of $F$.  The factor multiplying the identity matrix to give $C_x$, so that this factor is the (assumed uniform) variance of each component of the signal random vector $x$, then follows once $\gamma$ is computed.  Hence, from a statistical standpoint,  (\ref{7/7/20.3}), (\ref{7/8/20.9}) can be re-interpreted as implying a method resulting in {\it simultaneous} estimation of $x$ and the presumably uniform signal component variances $\epsilon^2/\gamma$. 
 
But as long as one is simultaneously estimating the solution and a covariance matrix feature, why stop there?   Instead of assuming that the   signal covariance matrix is proportional to the identity and estimating the proportionality constant simultaneously with the signal, one might only assume that the signal covariance matrix is diagonal in the basis defined by the eigenvectors of $F'F$, and seek to estimate the diagonal elements simultaneously with the signal (i.e., one assumes  that the matrix whose columns are the right singular vectors of  $F$ diagonalizes the signal covariance matrix).
The latter basis, $\{v_i\}$, is the natural one to be used in the context of this type of assumption since, considering the underlying setting where $F$ is a compact operator, one thing we {\it know} is that for the inverse problem to make sense the signal must be in the domain of $F$ (the problem is {\it defined} by $F$ operating on the signal). That is, we already know that the sum of the squares of the signal components in this basis is finite - and (absent other prior information) this is the only basis for which we know that to be true.  So, the assumption is simply that the components of random vector $x$ in this basis are independent, rather than additionally assuming that they all have the same variance (which would correspond to the assumption that the signal covariance matrix is proportional to $I$).

Proceeding in this way, one notes that the value of the $i$-th eigenvalue of $C_x$ is then the autocovariance of the $i$-th component of the signal random vector in the $\{v_i\}$-basis, i.e., the expectation of the square of the $i$-th component of random vector  $x$ in that basis (which is the expectation of $(v_i'x)^2$).   Thus, an alternative approach to the statistical treatment of (\ref{7/7/20.9}), (\ref{7/7/20.3}) (which we do not claim to be optimal in the solitary estimation of $x$)  would be to require that the eigenvalues of $C_x$ be the squares of the components of the solution estimate vector -  inspired by the actual meaning of those covariance matrix values as expectations  of the squares of the signal component random variables in the $\{v_i\}$-basis.  Such a solution estimate would then be  a fixed point of the dynamical system on $\mathbb{R}^m$ defined by iteration of the operator, \begin{equation}\label{10/19/10.3} A_\epsilon[w] \equiv \lim_{\alpha\rightarrow 0}\left(F'F+\epsilon^2\, V\mbox{\rm diag}\{(v_i'w)^2 + \alpha^2\}^{-1}V' \right)^{-1} F'y^\delta, \end{equation} where $V$ is the matrix whose $i$-th column is $v_i$ for each $i$. The role of $\alpha$ and $\lim_{\alpha\rightarrow 0}$ is to account for the possibility that one of the $v_i'w$ is zero.  To get the iterations started,  a solution estimate candidate $w$ (even chosen arbitrarily) is associated with a signal covariance matrix candidate $V \mbox{\rm diag}\{(v _i'w)^2\}V'$.    According to (\ref{7/8/20.5}),  $A_\epsilon[w]$ is then an updated solution estimate using the latter signal covariance matrix.  For the next iteration, the squares of the components of this updated solution estimate  $A_\epsilon[w]$ in the $\{v_i\}$-basis are used to define the entries of an updated version of the diagonal matrix on the right-hand-side of (\ref{10/19/10.3}) (i.e., updating the signal covariance matrix), which then supplies a successive updated solution estimate as $A_\epsilon[A_\epsilon[w]]$, etc. This iteration process can be repeated indefinitely.  A fixed point of the dynamical system evidently supplies a covariance matrix whose eigenvalues are the squares of the solution estimate components in the $\{v_i\}$-basis.  
 The fixed point $x=0$ is associated with a signal covariance matrix as the matrix of zeros. Thus, one would not select $w=0$ as an initial guess for the solution estimate to begin the iteration process (also, $w=0$ is very far from satisfying (\ref{7/7/20.3})).

But it is helpful that one can easily identify the fixed points of $A_\epsilon$. Inserting the previously noted
singular system associated with $F$ into the right-hand-side of (\ref{10/19/10.3}),  we obtain \begin{equation}\label{3/31/12.3} A_\epsilon[w] = \sum_i \frac{\sigma_i (u_i'y^\delta)}{\sigma_i^2 + \epsilon^2/(v_i'w)^2}v_i .\end{equation} A fixed point $p$ must satisfy  \begin{equation}\label{4/1/12.1} A_\epsilon[p] = \sum_i  \frac{\sigma_i (u_i'y^\delta)}{\sigma_i^2 + \epsilon^2/(v_i'p)^2}v_i = p = \sum_i  (v_i' p)v_i.\end{equation}  Since the coefficients of $v_i$ in the two sums in (\ref{4/1/12.1}) are necessarily equal, for $v_i'p\ne 0$ we obtain \begin{equation}\label{7/11/10.1} \sigma_iv_i'p = \frac{u_i' y^\delta \pm \sqrt{(u_i'y^\delta)^2 -4\epsilon^2}}{2}.\end{equation}  In the first place, this means that the series for a fixed point will only contain terms $i$ such that $(u_i'y^\delta)^2 >4\epsilon^2$.  Second, using Taylor series, it is easy to show that if the sign in front of the radical on the right-hand-side above is not the same as the sign of $u_i'y^\delta$, then $|\sigma_i v_i'p| =  O(\epsilon^2/|u_i'y^\delta|)$.  Indeed, whatever signal
 was contained in $u_i'y^\delta$ will have then been mostly 
 removed from the right-hand-side of (\ref{7/11/10.1}) (at least until $|u_i'y^\delta|$ is on the order of $\epsilon$, at which point in practical settings the data is mostly noise-dominated).   Thus, any fixed point of interest will be such that the sign in front of the radical is the same as the sign of $u_i'y^\delta$ .   The solution estimate that takes advantage of the greatest number of data subspaces available to this method is then,
\begin{equation}\label{7/19/10.1} p = \frac{1}{2}\sum_{\{i:\,|u_i'y^\delta|>2\epsilon\}} \left(u_i' y^\delta + {\rm sgn}(u_i' y^\delta) \sqrt{(u_i'y^\delta)^2 -4\epsilon^2}\right)\frac{v_i}{\sigma_i}.\end{equation}      
Note that the fixed points of of $A_\epsilon$ are truncations - a feature this solution estimate shares with the (quite different)  commonly used Truncated Singular Value Decomposition (TSVD) solution estimate \cite{vogel:2002}.  

The following result answers an obvious question of interest.

\begin{lemma} \label{7/6/21.9} Consider the solution estimate (\ref{7/19/10.1}) but with removal of the terms of the sum on the right-hand-side for which $|u_i'y^\delta|<4\epsilon$.  This is an attracting fixed point of the dynamical system resulting from repeated iterations of $A_\epsilon$. \end{lemma}

\begin{proof}  For fixed point $p$ to be attracting, it is sufficient that $A_\epsilon$ be continuously differentiable, with $\|\mbox{\bf D}\big[A_\epsilon[p]\big]\|_2 <1$, where $\mbox{\bf D}$ is the operator derivative and $\|\cdot\|_2$ denotes the operator norm subject to the Euclidean norm on the domain of $F$.
Using (\ref{3/31/12.3}), we obtain
\begin{equation}  \mbox{\bf D}\big[A_\epsilon[w]\big]  = \sum_i  \left( \frac{2\epsilon^2 \, (\sigma_i v_i'w)(u_i'y^\delta)}{(\sigma_i^2(v_i'w)^2+\epsilon^2)^2}\right)v_i \, v_i',\label{11/13/09.2} \end{equation} where $v_i v_i'$ indicates the outer product of $v_i$ with itself. Inspection of the terms of the finite sum on the right-hand-side of (\ref{11/13/09.2}) indicates that $A_\epsilon$ is continuously differentiable.   Thus, for $p$ to be attracting it is sufficient  that for each $i$, \begin{equation}\label{4/22/09.1} 2\epsilon^2 \, |(\sigma_i v_i'p)( u_i'y^\delta)|<(\sigma_i^2(v_i'p)^2+\epsilon^2)^2.\end{equation}
From (\ref{7/11/10.1}), \begin{equation} \label{7/17/20.1}\frac{1}{2}|u_i'y^\delta| < |\sigma_iv_i'p|,\end{equation} based on the requirement that the sign of the radical match  the sign of $u_i'y^\delta$. Thus, it is sufficient that the inequality (\ref{4/22/09.1}) hold when on the left-hand-side $u_i'y^\delta$ is replaced by $2\sigma_i v_i'p$, and on
 the right-hand-side $\epsilon^2$ is replaced by $0$. One then obtains $4\epsilon^2 < (\sigma_i v_i'p)^2$ as a sufficient condition for $p$ to be attracting.  Applying (\ref{7/17/20.1}) again, one notes that it is sufficient if $4\epsilon^2 < (1/4)(u_i'y^\delta)^2$, i.e., $|u_i'y^\delta)|>4\epsilon$.
\end{proof}

For the purposes of this paper, the point in presenting all this is that the solution estimate $p$ given by (\ref{7/19/10.1}) is a critical point of the {\it geometric mean} of the solution estimate component absolute values on an implied ellipsoid  derived from the discrepancy ellipsoid on which $p$ lies - in contrast to the most commonly used solution estimate, which is a critical point of the Euclidean norm, i.e., a critical point of the {\it quadratic mean}, on the discrepancy ellipsoid on which it lies.  That is, the geometric mean has both a statistical and deterministic interpretation in the context of the regularization of inverse problems, just as the Euclidean norm does. 
 
 \begin{theorem} \label{7/6/21.6} Consider \begin{equation} \tag{\ref{7/7/20.9}} y^\delta = F x +\nu, \end{equation} where $y^\delta$ is known, $F:\mathbb{R}^k\rightarrow\mathbb{R}^k$ is a given nonsingular linear transformation associated with singular value decomposition matrix $U\Sigma V'$, and $\{u_i\}$ and $\{v_i\}$ are the sets of columns of $U$ and $V$, respectively.  Further consider the dynamical system on $\mathbb{R}^k$ defined by iteration of the operator, \begin{equation}\tag{\ref{10/19/10.3}} A_\epsilon[w] \equiv \lim_{\alpha\rightarrow 0}\left(F'F+\epsilon^2\, V\mbox{\rm diag}\{(v_i'w)^2 + \alpha^2\}^{-1}V' \right)^{-1} F'y^\delta. \end{equation} 
As a solution estimate for $x$, the fixed point $p_\epsilon$ of largest Euclidean norm of this dynamical system  has the following interpretations:
\begin{enumerate}\item {\rm Statistical:} If $\epsilon^2$ is the variance of zero-mean white Gaussian noise random vector $\nu$, then 
$p_\epsilon$ is the  {\rm maximum a posteriori} estimate for zero-mean Gaussian signal random vector $x$ of greatest Euclidean norm
 that results from a signal covariance matrix  whose eigenvectors are the members of the basis $\{v_i\}$ and with the additional feature that its eigenvalues are the squares of the components of the solution estimate with respect to those basis vectors.  \item  
{\rm Deterministic:} Denote the discrepancy of $p_\epsilon$ as $\delta_\epsilon \equiv \| y^\delta - Fp_\epsilon\|$, which may be adjusted by varying $\epsilon$ so as to allow for the influence of prior knowledge concerning the magnitude of unknown noise vector $\nu$.  Let $\mathbb{V}_\epsilon \subset \mathbb{R}^k$ be the space spanned by the subset of eigenvectors $\{v_i: |u_i'y^\delta| > 2\epsilon \}$.  
Consider the ellipsoid \[ \partial\Omega \equiv \{ z\in \mathbb{V}_\epsilon: \| y^\delta - Fz\|^2 = \delta_\epsilon^2 \}.\]   Then $p_\epsilon\in \mathbb{V}_\epsilon$ is a critical point of the geometric mean of the absolute values of the coordinates of the points of $\partial\Omega$ with respect to the above eigenvector subset basis.
\end{enumerate}
\end{theorem}

\begin{proof}
The statistical interpretation follows from the prior observations related to the form of the right-hand-side of (\ref{7/8/20.5}) and inspection of the form of (\ref{10/19/10.3}) as regards a fixed point.

We now consider the deterministic interpretation.   
To simplify our notation, without loss we assume an arrangement of the sequence of right singular vectors of $F$ such that $\{v_i\}_{i=1}^n$ is the set $\{v_i: |u_i'y^\delta| > 2\epsilon \}$, and we will write $p_\epsilon$ as simply $p$.  We observe that $p$ is a member of $\mathbb{V}_\epsilon$, since we have previously shown that $p$ is given by (\ref{7/19/10.1}). 
Define $\hat{p} \equiv (v_1'p,\dots,v_n'p)$, i.e., $\hat{p}$ is the coordinate expression of $p$ as a member of $\partial\Omega\subset \mathbb{V}_e$ (thus, $\hat{p}$ is $p$ with suppression of the latter's null-valued coordinates with respect to the full basis of $\mathbb{R}^k$ given by $\{v_i\}$). 
Define $C_p \equiv \mbox{diag}\{(v_i'p)^2\}$ and  $C_{\hat{p}} \equiv  \mbox{diag}\{(v_i'p)^2\}_{i=1}^n$.  Since $p$ is a fixed point of $A_\epsilon$,  (\ref{10/19/10.3}) implies \begin{equation} p = \lim_{\alpha\rightarrow 0} \left(F'F+\epsilon^2\,\left[C_p + \mbox{diag}\{\alpha^2\}\right]^{-1}\right)^{-1} F'y^\delta.\label{12/14/20.1}\end{equation} However, the null coordinates of $p$ correspond to the null diagonal values of $C_p$.  There is no intrinsic loss in suppressing these null entries.  {\it Hence, the right-hand-side of (\ref{12/14/20.1}), can be cast as the solution to the variational problem of finding the point minimizing $s'C^{-1}_{\hat{p}} s$ over $z\in \partial\Omega \subset \mathbb{R}^n$, which is $\hat{p}$}.
Specifically, (\ref{12/14/20.1}) implies, and is replaced by, 
\begin{equation} \hat{p} = \left(\hat{F}'\hat{F}+\epsilon^2\,C_{\hat{p}}^{-1}\right)^{-1} \hat{F}'y^\delta, \label{12/19/20.1}\end{equation} where $\hat{F}$ is the operator resulting from the restriction of the domain of $F$ to $\mathbb{V}_\epsilon$.
  According to the principle of Lagrange multipliers, the point minimizing $s'C^{-1}_{\hat{p}} s$ on \[ \partial\Omega = \{s\in \mathbb{V}_\epsilon: \| \hat{F} s -y^\delta\|^2 = \delta_\epsilon^2\} = \{s\in \mathbb{V}_\epsilon: (s'\hat{F}' - (y^{\delta})')(\hat{F} s - y^\delta) = \delta_\epsilon^2\},\] is the critical point of, \[ \mathcal{L}(s,\lambda) \equiv \lambda\left[ (s'\hat{F}' - (y^{\delta})')(\hat{F} s - y^\delta) -\delta_\epsilon^2\right] + s'C^{-1}_{\hat{p}} s.\] 
Setting to zero all of the partial derivatives of $\mathcal{L}$ with respect to the components of $s$ leads to the solution being given as the right-hand-side of (\ref{12/19/20.1}), where the factor $\epsilon^2$ is the reciprocal of the Lagrange multiplier.  The latter is determined by the final equation $\frac{\partial \mathcal{L}}{\partial \lambda} = 0$, which places the solution on $\partial\Omega$.  This solution is indeed $\hat{p}$ since we are {\it given} (\ref{12/19/20.1}) (it is implied by (\ref{12/14/20.1})) and we know that $\hat{p}$ lies on $\partial\Omega$. 

Now define the function  $h:\mathbb{V}_\epsilon\rightarrow \mathbb{R}$ with $h(s) \equiv s'C^{-1}_{\hat{p}} s$. According to the prior paragraph, $\hat{p}$ minimizes $h(z)$ on  $\partial\Omega$.   Consider the level set of $h$ given by the ellipsoid, \[ \mathcal{H} \equiv \{s\in \mathbb{V}_\epsilon: s'C_{\hat{p}}^{-1} s = n \} = \left\{s\in\mathbb{V}_\epsilon: \sum_{i=1}^n \frac{s_i^2}{p_i^2} = n\right\},\] where $s_j\equiv v_j's$ and  $p_j\equiv v_j'p$ for $v_j\in \{v_i\}_{i=1}^n$.  
The point $\hat{p}$ is evidently a member of $\mathcal{H}$.  The tangent space of ellipsoid $\mathcal{H}$ at $\hat{p}$  has normal vectors proportional to the gradient of $h$ at $\hat{p}$, and with respect to the $\{v_i\}_{i=1}^n$-basis,
\begin{equation}\label{3/18/12.1}  \nabla h(s)\Big|_{s=\hat{p}} = \,\, 2\left(\frac{1}{p_1},\dots, \frac{1}{p_n}\right).\end{equation} But as we have noted above, $\hat{p}$ is the solution estimate minimizing $s'C^{-1}_{\hat{p}} s$ on $\partial\Omega$.  Thus, $\mathcal{H}$ is tangent to $\partial\Omega$ at $\hat{p}$.  
 
We next introduce  $g:\mathbb{V}_\epsilon\rightarrow \mathbb{R}$ with $g(s) = \prod_{i=1}^n |s_i|$ (the $n$-th power of the geometric mean of the absolute values of a point's components with respect to the $\{v_i\}_{i=1}^n$-basis).  The level set of $g$ that contains $\hat{p}$ is,  \[\mathcal{G} \equiv \left\{ s\in \mathbb{V}_\epsilon:\prod_{i=1}^n |s_i| = g(\hat{p})\right\}.\] We observe that at any point of this level set,
   \begin{equation}\label{9/18/22.1} \frac{\partial g}{\partial s_j} = {\rm sgn}(s_j) \prod_{i\ne j} |s_i| = \frac{g(\hat{p})}{s_j}.\end{equation} The tangent space of $\mathcal{G}$ at $\hat{p}$ has normal vectors proportional to the gradient of $g(s)$ at $\hat{p}$ and, with respect to the $\{v_i\}_{i=1}^n$-basis, (\ref{9/18/22.1}) implies \begin{equation}\label{9/15/20.1} \nabla g(s)\Big|_{s=\hat{p}} = \,\, g(\hat{p}) \left(\frac{1}{p_1},\dots, \frac{1}{p_n}\right),\end{equation} which according to (\ref{3/18/12.1}) is proportional to $\nabla h(s)\big|_{s=\hat{p}}$.
Hence, $\mathcal{H}$ and $\mathcal{G}$ both contain $\hat{p}$, and share the same tangent space at $\hat{p}$, i.e., $\mathcal{G}$ and $\mathcal{H}$ are tangent at $\hat{p}$.  But we have seen that $\mathcal{H}$ is tangent to $\partial\Omega$ at $\hat{p}$.  Hence $\mathcal{G}$ is tangent to $\partial\Omega$ at $\hat{p}$, and $\hat{p}$ is thereby a critical point of $g$ on $\partial\Omega$.  The $n$-th root of $g$, which is the geometric mean, then also has a critical point on $\partial\Omega$ at $\hat{p}$, which is equivalent to the final statement in the theorem. 
 \end{proof}

 \section{The regularization method motivates the program of \cite{fgreensite:2023}}
 
Theorem \ref{7/6/21.6} indicates that the regularized solution estimate is a critical point of the geometric mean on a discrepancy manifold.  This greatly contrasts with zero-order Tikhonov regularization, where the regularized solution estimate is a critical point of the Euclidean norm on a discrepancy manifold (for Tikhonov regularization in general, the procedure is at least performed with seminorms).  So we are already motivated to think of the geometric mean as some kind of a ``norm", albeit one not satisfying the triangle inequality. 

The geometric mean pertains to this regularization method precisely because of (\ref{9/18/22.1}), (\ref{9/15/20.1}), which rely on on $\nabla g(s) = s^{-1}g(s)$, wherein $s=(s_1,\dots,s_n)$, $g(s)= \prod_i |s_i|$ (the geometric mean raised to the $n$-th power), and \begin{equation}\label{12/5/23.1} s^{-1}\equiv\left(\frac{1}{s_1},\dots, \frac{1}{s_n}\right).\end{equation} 
The geometric mean is a degree-1 positive homogeneous function extracted from $g(s) $, but we will modify it slightly by multiplying it by a constant for reasons that will be clear shortly.   Thus, for a point $s = (s_1, \dots, s_n)\in \mathbb{R}^n$, with $s_i\ne 0$ for $i=1,\dots,n$, we introduce  \begin{equation}\label{12/6/23.3} \ell(s)\equiv\sqrt{n}\left(\prod_i |s_i|\right)^{1/n}.\end{equation}   Combining this with (\ref{12/5/23.1}), we have \begin{equation}\label{12/6/23.1}\ell(s)\ell(s^{-1}) = n = \mathbf{1}\cdot\mathbf{1}\equiv \|\mathbf{1}\|^2,\end{equation} where $\mathbf{1}\equiv (1,1,\dots,1)$.  For $s$ in a neighborhood of $\mathbf{1}$,
\[ \nabla \ell(s) = \frac{\ell(s)}{n}\left(\frac{1}{s_1},\dots,\frac{1}{s_n}\right) = \frac{s^{-1}}{\ell(s^{-1})},  \]i.e., \begin{equation}\label{12/5/23.2} s^{-1} = \ell(s^{-1})\nabla\ell(s),\end{equation}  Applying $\ell$ to both sides of the above equation, and using the degree-1 positive homogeneity of $\ell$,  we obtain \begin{equation}\label{12/6/23.5} \ell(\nabla\ell(s))
= 1.\end{equation}

 Equations (\ref{12/5/23.2}), (\ref{12/6/23.5}), describing the geometric mean, look a lot like the equations describing the Euclidean norm,  \begin{equation}\label{11/29/23.5}  s = \ell(s)\nabla\ell(s),\end{equation} and \begin{equation}\label{11/29/23.6} \ell(\nabla\ell(s)) = 1.\end{equation}  That is, the solution $\ell(s)$ to (\ref{11/29/23.5}), (\ref{11/29/23.6}) is the Euclidean norm, and the two equations express a point as the product of its Euclidean norm with an associated unit-norm direction.
So from this standpoint also, it seems that the geometric mean {\it is} indeed rather ``norm-like".  In particular, for the case where $\ell(s)$ derives from the geometric mean as in (\ref{12/6/23.3}) (which, in the above regularization context, we have already come to think of as some kind of a ``norm"), according to (\ref{12/5/23.2}), (\ref{12/6/23.5}) a point in the neighborhood of $\mathbf{1}$ has the decomposition,
\begin{equation}\label{8/4/20.2} s = \ell(s)\nabla \ell(s^{-1}) \mbox{  with   }  \ell(\nabla \ell(s^{-1})) = 1,\end{equation} where the expression $\nabla \ell(s^{-1})$ means that the gradient is obtained following which it is evaluated at $s^{-1}$.  Like the pair of equations (\ref{11/29/23.5}), (\ref{11/29/23.6}) for the Euclidean norm, the equations (\ref{8/4/20.2}) can also be interpreted as expressing a point as the product of its ``norm" with an associated unit-``norm"  direction.

Looking at (\ref{12/6/23.1}), (\ref{12/5/23.2}),  (\ref{12/6/23.5}) in isolation, we have to wonder: what is the significance of the ``inverse of a point", $s^{-1}$?  Well, it suggests that $s$ is an element of a unital algebra.  Which algebra?  In the case at hand, the inverse of a point happens to be given by (\ref{12/5/23.1}).   Equation (\ref{12/5/23.1}) indicates that the multiplicative identity element is ${\bf 1}= (1,1,\dots, 1)$, since this element will be equal to its inverse.  But we must have  $ss^{-1} = (s_1,\dots,s_n)(1/s_1,\dots,1/s_n) = {\bf 1}$, where element juxtaposition denotes the product. The equations of the last two sentences indicate that we must then be dealing with the algebra (whose vector space of elements is $\mathbb{R}^n$) defined by component-wise addition and multiplication, i.e., $\bigoplus_{i=1}^n \mathbb{R}$.    And that gets one thinking: how many other algebras are associated with an entity $\ell(s)$ satisfying something like (\ref{12/6/23.1}), (\ref{12/5/23.2}),  (\ref{12/6/23.5}), i.e., satisfying ``norm-like" equations respecting the multiplicative inversion operation of the algebra?
 
The answer is not many.  This is because, given (\ref{12/6/23.1}) (which mandates respect for the multiplicative inversion operation), the right-hand-side of (\ref{12/5/23.2}) is a gradient, i.e., $\nabla\log \ell(s)$ - but the left-hand-side, $s^{-1}$,  is not a gradient for most unital algebras.  
To rectify that, we would have to replace $s^{-1}$ on the left-hand-side of (\ref{12/5/23.2}) with $Ls^{-1}$, where $L$ has the feature that $Ls^{-1}$ {\it is} a gradient (and by the way, the left-hand-side of (\ref{12/6/23.5}) would then become $\ell(L^{-1}\nabla\ell(s))$, assuming $L$ to be nonsingular).  

The requirement that $Ls^{-1}$ be a gradient means that $L$ must satisfy, \[ d([Ls^{-1}]\cdot ds) = 0.\]  The program developed in \cite{fgreensite:2023} proceeds from that equation, presented in the context of Definition 2.1 of that paper.  

\section{Extension of the regularization method to compact operators}

The following theorem establishes that the above ``geometric mean methodology" is indeed a regularization method, i.e., convergent for suitable selection of a regularization parameter based on knowledge of the noise magnitude.    But we hasten to add that although the geometric mean application is associated with a regularization method that works well in practice, the latter is {\it not} asserted to have efficacy comparable to that of e.g., zero-order Tikhonov regularization.  In fact, derivation of convergence rates for typical subsets of the domain of the relevant compact operator is problematic.  

\begin{theorem} \label{7/6/21.7} Let $F$ be a given linear compact operator with singular value expansion $F =  \sum_i \sigma_i u_i v_i'$, and consider, \begin{equation} \tag{\ref{7/7/20.9}} y^\delta = F x +\nu, \end{equation} where $y^\delta$ in the range of $F$ is known, but $x$ in the domain of $F$ and $\nu$ in the range of $F$ are unknown.  Suppose also that,  \begin{equation}\tag{\ref{7/12/20.1}}  \| Fx-y^\delta\|^2 \le \delta^2,\end{equation} where $\delta$ is known.   On the domain of $F$, the greatest-magnitude fixed point of the operator,\begin{equation}\label{10/17/20.1}A_\epsilon[w] \equiv \lim_{\alpha\rightarrow 0}\left(F'F+\epsilon^2\, 
\sum_i \frac{v_iv_i'}{(v_i'w)^2+\alpha^2}
\right)^{-1} F'y^\delta, \end{equation} is given by \begin{equation}\label{7/15/20.2} x^\delta = \frac{1}{2}\sum_{\{i:\,|u_i'y^\delta|>2\epsilon\}} \left(u_i' y^\delta + {\rm sgn}(u_i' y^\delta) \sqrt{(u_i'y^\delta)^2 -4\epsilon^2}\right)\frac{v_i}{\sigma_i}.\end{equation} 
If we specify $\epsilon = \epsilon(\delta) = \delta$, then for $x^\dagger \equiv F^\dagger F x$  we have $\lim_{\epsilon(\delta)\rightarrow 0} \| x^\dagger - x^\delta \| =0 $.
\end{theorem}
  
  \begin{proof} In the following argument we will assume that there are no vanishing singular values of $F$, so that $x^\dagger = x$.  If there {\it are} vanishing singular values, the argument is easily modified by appropriately exchanging $x$ for $x^\dagger$ as necesary.
  
 The prior derivation of (\ref{7/19/10.1}) holds regardless of whether $F$ is a finite-dimensional matrix or a linear compact operator, so (\ref{7/15/20.2}) follows with $p$ replaced by $x^\delta$. The error of the solution estimate is thus,  \begin{eqnarray}\nonumber  x - x^\delta &=&  \sum_{\{i:\,|u_i'y^\delta|\le 2\epsilon(\delta)\}}(v_i' x)v_i \,\, + \sum_{\{i: |u_i'y^\delta)| > 2\epsilon(\delta) \}}  \frac{(\sigma_i v_i' x - u_i' y^\delta)}{\sigma_i}v_i \\ 
  && + \sum_{\{i:\,|u_i'y^\delta| > 2\epsilon(\delta) \}} \frac{1}{2}\bigg(u_i'y^\delta - {\rm sgn}(u_i' y^\delta) \sqrt{(u_i'y^\delta)^2 -4\epsilon(\delta)^2}
\bigg)\frac{1}{\sigma_i}v_i, \label{7/19/20.1}
\end{eqnarray} where the first sum above is the ``truncation error" and the last two sums are the ``stability error".
It is now simply a matter of evaluating the norms of the three sums on the right-hand-side above, as the norm of the solution estimate error will be less than the sum of those norms.  
  
The square of the norm of the first sum on the right-hand-side of (\ref{7/19/20.1}) is, \[  \sum_{\{i:\,|u_i'y^\delta|\le 2\epsilon(\delta)\}}(v_i' x)^2 \,\, =\,\,   \sum_{i=1}^\infty (v_i' x)^2 H\big(2\epsilon(\delta) - |u_i'y^\delta|\big)\,\,  \le \,\, \| x\|^2,\] where $H(\cdot)$ is the Heaviside function.  The Dominated Convergence Theorem can then be invoked so that application of $\lim_{\epsilon(\delta)\rightarrow 0}$ to the sum on the left-hand-side above is seen to be equal to the sum over the individual limits of the terms of the sum on the right-hand-side of the first equality.     Since the limit of each of those terms is zero, the norm of the first sum on the right-hand-side of (\ref{7/19/20.1}) vanishes in the limit as $\epsilon(\delta) = \delta \rightarrow 0$.  
 
Now, both the second and third sums on the right-hand-side of (\ref{7/19/20.1}) are over terms with $i$ such that $ |u_i'y^\delta)| > 2\epsilon(\delta)$.
 But since $\epsilon(\delta) = \delta$, we have,  \[ 2\delta = 2\epsilon(\delta) < |u_i'y^\delta| = |\sigma_i v_i'x + u_i' \nu| \le |\sigma_i v_i'x| + \delta,\] with the final inequality above following from (\ref{7/12/20.1}). Consequently, for both the second and third sums on the right-hand-side of (\ref{7/19/20.1}), \begin{equation}\label{7/15/20.3} \frac{\delta}{\sigma_i} < |v_i' x|.\end{equation}
  So we can apply this to the second sum on the right-hand-side of (\ref{7/19/20.1}) to obtain, \begin{equation}\label{7/16/20.5} \sum_{\{i:\,|u_i'y^\delta|>2\epsilon(\delta)\}}  \frac{(\sigma_iv_i' x - u_i' y^\delta)^2}{\sigma_i^2} < \frac{\delta^2}{\sigma_{m-1}^2} <  (v_{m-1}'x)^2, \end{equation} 
  where $\sigma_{m-1}$ is the smallest singular value appearing in the sum.   
The first inequality follows from (\ref{7/12/20.1}) (i.e., (\ref{7/12/20.1}) implies that the above sum over just the numerators of the included terms is $\le \delta^2$).   As $\epsilon(\delta) = \delta\rightarrow 0$, it follows that $\sigma_{m-1}\rightarrow 0$ (since $F$ is compact), i.e., $m\rightarrow \infty$.  But $\lim_{m\rightarrow \infty} (v_{m-1}'x)^2 =0$.  Equation (\ref{7/16/20.5}) then implies that the norm of the second sum on the right-hand-side of  (\ref{7/19/20.1}) tends to zero as $\epsilon(\delta) = \delta \rightarrow 0$. 
  
So now we consider the third sum on the right-hand-side of  (\ref{7/19/20.1}).
Note that $0< 1-\sqrt{1-z^2} <z^2$ for $z\in (-1,1)$. Since each term of the sum is such that $ |u_i'y^\delta)| > 2\epsilon(\delta)$, it follows that,
\begin{equation}\label{7/22/20.1} \left| u_i' y^\delta - {\rm sgn}(u_i' y^\delta) \sqrt{(u_i'y^\delta)^2 -4\epsilon(\delta)^2} \right| <  \frac{[2\epsilon(\delta)]^2}{|u_i'y^\delta|} <  2\epsilon(\delta).\end{equation}
Applying this to the square of the norm of the third sum on the right-hand-side of  (\ref{7/19/20.1}), we obtain that squared norm as
\begin{eqnarray} \nonumber   \sum_{i=1}^\infty  \frac{1}{4\sigma_i^2}\left(u_i' y^\delta - {\rm sgn}(u_i' y^\delta) \sqrt{(u_i'y^\delta)^2 -4\epsilon(\delta)^2}\right)^2 \Big(1-H\big(2\epsilon(\delta) - |u_i'y^\delta|\big)\Big)  \\   <  \,\,\,\, \sum_{i=1}^\infty   \left(\frac{\epsilon(\delta)}{\sigma_i}\right)^2 \Big(1-H\big(2\epsilon(\delta) - |u_i'y^\delta|\big)\Big) \,\,\,  < \,\,\,\, \sum_{i=1}^\infty  (v_i'x)^2  = \|x\|^2. \label{7/24/20.6}\end{eqnarray} 
  The second inequality in (\ref{7/24/20.6})  follows from (\ref{7/15/20.3}), since $\epsilon(\delta)=\delta$.   The Dominated Convergence Theorem then implies that $\lim_{\epsilon(\delta)\rightarrow 0}$ can be taken inside the  sum on the left-hand-side of  (\ref{7/24/20.6}) in evaluation of the limit of that sum. The limit of each of the terms of the sum is zero, indicating that the limit of the norm of the third sum on the right-hand-side of  (\ref{7/19/20.1}) is zero as $\epsilon(\delta) = \delta\rightarrow 0$.  
 
 Since the limit of the norm of each of the sums in  (\ref{7/19/20.1}) is zero as $\epsilon(\delta) = \delta \rightarrow 0$, it follows that $\lim_{\epsilon(\delta)\rightarrow 0} \| x - x^\delta\| = 0$.  \end{proof}
 
 \section{An example of application of the regularization method}

 The figure on the following page shows the results of a numerical simulation, computed in MatLab.
 
   The source ``$x$" is a discetized time series of spikes in 100 data points, shown at the (1,1) position in the $(3\times 3)$-matix of tracings in the figure.  The source is smoothed by convolution with a Gaussian density and then some simulated white noise is added, giving the measurements  $``y+\text{noise}"$ time series at the (2,1) position of the figure.  The zero-order Tikhonov regularized signal estimate corresponding to the L-curve corner regularization parameter selection method \cite{hansen:1992} is shown at the figure's (1,2) position.  The alternative method presented in this paper, also using the L-curve corner regularization parameter selection method (for selection of $\epsilon$ in Theorem \ref{7/6/21.6}), is at the (1,3) position [in the caption, ``orthog" relates to the method's use of an orthogonal operator rather than a seminorm as in Tikhonov regularization; ``zero-order" refers to the use of the eigenvector matrix of $F'F$ for this purpose (where $F$ is the transfer matrix as in Theorem \ref{7/6/21.6}), rather than another choice based on some degree of prior knowledge of the source features.   The L-curve for the zero-order Tikhonov method and the method of Theorem \ref{7/6/21.6} are as indicated in the (2,2) and (2,3) positions, respectively.  The L-curves are generated over 30 orders of magnitude of regularization parameter value, with the solution estimate relative errors for each value of the regularization parameter shown at the respective (3,2) and (3,3) positions.  The latter depict the relative error of the methods for each value of regularization parameter, where relative error is defined by $\frac{\|\text{source - source estimate}\|}{\|\text{source}\|}$, and $\|\cdot\|$ is the Euclidean norm on $\mathbb{R}^n$ where $n$ is the number of data points comprising $x$ (i.e., $n=100$ here).  The plot at the (3,1) position is a scatter plot of (abscissa, ordinate) of the L-curve corner, for the L-curve generated at each iteration step of the new method (so, the plot at the (2,3) position is the full L-curve for the final iteration step, and its corner is one of the points of the scatter plot at the (3,1) position).  The points associated with successive iterations are connected by a dotted line.
 
 \hspace{-2cm} \includegraphics{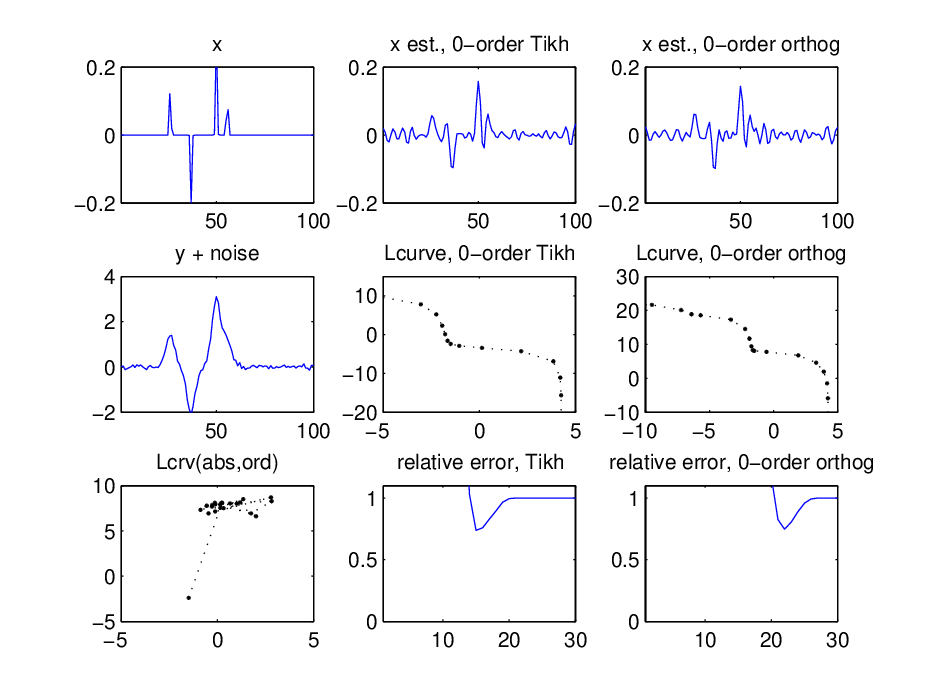}

 The quality of the new method's solution estimate is not quite as good as the estimate provided by zero-order Tihonov regularization, as is to be expected since zero-order Tikhonov regularization is essentially optimal under conditions of this simulation.  Nevertheless, the new method's solution estimate is respectable and, as indicated by Theorem \ref{7/6/21.7}, it represents a convergent method.


\begin{thebibliography}{99}
  
\bibitem{fgreensite:2023} Greensite F (2023) Novel isomorphism invariants of real algebras, arXiv:2306.14995.
  
       \bibitem{fgreensite:2022b} Greensite, F (2022), A new proof of the Pythagorean Theorem and generalization of the usual algebra norm. arXiv:2209.14119 
  
      \bibitem{vogel:2002} Vogel C (2002) Computational Methods for Inverse Problems. Philadelphia: SIAM.
      
      \bibitem{engl:2000}  Engl HW,  Hanke M,  Neubauer A (2000) Regularization of Inverse Problems, Klewer, Boston.
      
      \bibitem{hansen:1992} Hansen PC (1992) Analysis of discrete ill-posed problems by means of the L-curve, SIAM Review, 34:563-580.
 
 \end{thebibliography}
  \end{document}